\documentclass[11pt,english]{article}
\usepackage[letterpaper]{geometry}
\usepackage{array}
\usepackage{amsthm}
\usepackage{amsmath}
\usepackage{setspace}
\usepackage{amssymb}
\usepackage{amsfonts}
\usepackage{graphicx}
\usepackage{bbm}
\usepackage{cleveref}
\crefname{lemma}{Lemma}{lemmas}
\crefname{claim}{Claim}{claims}
\crefname{corollary}{Corollary}{corollaries}
\crefname{theorem}{Theorem}{theorems}
\crefname{fact}{Fact}{facts}
\crefname{conjecture}{Conjecture}{conjectures}

\makeatletter

\newtheorem{conjecture}{Conjecture}
\newtheorem{theorem}[conjecture]{Theorem}

\newtheorem{lemma}[conjecture]{Lemma}

\newtheorem{fact}[conjecture]{Fact}

\newcommand{\fl}[1]{\ensuremath{\lfloor #1 \rfloor}}

\begin{document}

\title{Order-Preserving Freiman Isomorphisms}
		
\author{Gagik Amirkhanyan, Albert Bush, Ernie Croot}
\date{\today}
\maketitle

\begin{abstract}
An order-preserving Freiman 2-isomorphism is a map $\phi:X \rightarrow \mathbb{R}$ such that $\phi(a) < \phi(b)$ if and only if $a < b$ and $\phi(a)+\phi(b) = \phi(c)+\phi(d)$ if and only if $a+b=c+d$ for any $a,b,c,d \in X$.  We show that for any $A \subseteq \mathbb{Z}$, if $|A+A| \le K|A|$, then there exists a subset $A' \subseteq A$ such that the following holds: $|A'| \gg_K |A|$ and there exists an order-preserving Freiman 2-isomorphism $\phi: A' \rightarrow [-c|A|,c|A|] \cap \mathbb{Z}$ where $c$ depends only on $K$.  Several applications are also presented.
\end{abstract}

\section{Introduction}
Let $G$ and $H$ be additive groups, and let $A \subseteq G$ and $B \subseteq H$.  A Freiman $k$-homomorphism is a map $\phi:A \rightarrow B$ such that
\[ \phi(x_1) + \ldots + \phi(x_k) = \phi(y_1) + \ldots + \phi(y_k) \]
whenever
\[ x_1 + \ldots + x_k = y_1 + \ldots + y_k. \]
Such a map $\phi$ is called a Freiman $k$-isomorphism if the converse holds as well.  If $A$ and $B$ have an ordering, then $\phi$ is order-preserving when 
\[ \phi(a) < \phi(b) \text{ if and only if } a < b. \]
A Freiman $2$-isomorphism will frequently be referred to as just a Freiman isomorphism.  Freiman isomorphisms are used to transfer an additive set $A$ in some arbitrary abelian group $G$ into a more amenable ambient group or set (such as $\mathbb{R}$, $\mathbb{Z}_N$, or $[1,n]$) while preserving the additive structure of $A$.  Previously, finding such a mapping from $\mathbb{Z}_p$ to $\mathbb{Z}$ was referred to as a `rectification' principle.  Rectification principles were studied in \cite{billevruz98} and \cite{greruz06}.  However, previous studies were not sensitive to the ordering since the domain of the mapping was $\mathbb{Z}_p$ instead of $\mathbb{Z}$, and the co-domain was not necessarily an interval of size $O(|A|)$.  Moreover, such mappings proceeded by dilating the set in $\mathbb{Z}_p$ by a residue $a$ which implicitly comes from Minkowski's theorem, a Fourier analytic argument, or a probabilistic argument.  Hence, controlling the order requires a substantially different approach.  We refer the interested reader to Chapters 3 and 5 of \cite{taovu10} for a detailed exposition on the background and various uses of Freiman isomorphisms.

The main tool we introduce in this paper allows one to find an order-preserving Freiman isomorphism from a set of $n$ integers to the interval $[-cn, cn] \cap \mathbb{Z}$ where $c$ is not too large provided that the original set is additively structured.  We call this tool a `Condensing Lemma' since, in a sense, it allows one to view sets with small doubling as dense subsets of an interval.  Although similar theorems have been proved before (see theorem 1.4 \cite{greruz06}), we reiterate that previous studies did not take order-preservation into consideration.  
\begin{theorem}[Condensing Lemma]\label{condenselemma}
For any $K > 0$, there exists a $c_1, c_2$ such that if $A \subseteq \mathbb{Z}$ is such that $|A+A| \le K|A|$ then the following holds: there exists $A' \subseteq A$ with $|A'| \ge c_1 |A|$, and there exists an order-preserving Freiman 2-isomorphism $\phi:A' \rightarrow [-c_2 |A'|, c_2|A'|] \cap \mathbb{Z}$.
\end{theorem}
Since the constants $c_1$ and $c_2$ depend exponentially on $K$, we do not bother specifying their exact value.  In order to prove the Condensing Lemma, we need Freiman's Theorem \cite{fre66} which guarantees us a large, but low-dimensional generalized arithmetic progression $P$ containing $A$ when $A$ has a small doubling.  There have been many important improvements to Freiman's original version, and we refer the reader to the recent work by Sanders \cite{san13} who gives the best-known bounds for the constants $c_1$ and $c_2$ stated below.
\begin{theorem}[Freiman's Theorem]\label{sanders}
Suppose $A \subseteq \mathbb{Z}$ satisfies $|A+A| \le K|A|$.  Then, there exists absolute constants $c_1, c_2$ dependent only on $K$ such that $A$ is contained in a proper, symmetric, generalized arithmetic progression $G$ of dimension at most $c_1$ and size at most $c_2|A|$.
\end{theorem}
The proof of the Condensing Lemma consists of first applying Freiman's theorem so that we may approximate $A$ by a generalized arithmetic progression $G=\{ \sum_{i=1}^k x_i d_i: |x_i| \le L_i\}$.  Then, using elementary techniques from convex geometry, we show that there is a generalized arithmetic progression $G'=\{ \sum_{i=1}^k x_i d_i': |x_i| \le L_i/4 \}$ that shares the additive properties of $G$, but is contained in an interval of length $O(|G|)$.

After we prove the Condensing Lemma, we provide some applications.  Let $A = \{a_1 < a_2 < \ldots < a_n\}$ be a finite subset of the integers, and denote the \textit{indexed energy} of $A$ as
\begin{equation}\label{indexdef} EI(A) := \{(i,j,k,l): a_i + a_j = a_k + a_l \text{ and } i+j=k+l\}. \end{equation}
The reader may be more familiar with the additive energy of a set which can be used to control the size of the sumset:
\begin{equation}\label{energydef} E(A) = |\{(i,j,k,l): a_i + a_j = a_k + a_l\}| \ge \frac{|A|^4}{|A+A|}. \end{equation} 
We determine the precise relationship between $E(A)$ and $EI(A)$.  Although the indexed energy of a set has not been directly studied, the additive properties of a set and how they interact with the related indices has appeared in various forms.  Solymosi \cite{sol05a} studied the situation when $a_i + a_j \neq a_k + a_l$ for $i-j = k-l = c$ for a fixed constant $c$, and in particular when a set $A$ has the property that $a_{i+1} + a_i \neq a_{j+1} + a_j$ for all pairs $i,j$.  Brown et al \cite{brojunpoe14} asked if one finitely colors the integers $\{1, \ldots , n\}$, must one be forced to find a monochromatic `double' 3-term arithmetic progression $a_i + a_j = 2a_k$ where $i + j = 2k$?

\textbf{Layout and Notation. }In section 2, we state some basic notions from convex geometry, and then we prove the Condensing Lemma.  In section 3, we study the indexed energy of a set, providing both an extremal construction of a set with large additive energy and small indexed energy as well as proving a Balog-Szemer\'edi-Gowers type theorem to find a subset with large indexed energy.  Section 4 contains further applications and conjectures related to the Condensing Lemma as well as the indexed energy.  

The sumset is defined as $A+B := \{ a+b : a \in A, b\in B\}$.  We write $[a,b]$ for $[a,b] \cap \mathbb{Z}$, and similarly for $[a,b), (a,b)$, and $(a,b]$.  For two functions $f, g$, we write $f \gg g$ if $f(n) \ge c g(n)$ for some constant $c$ and $n$ sufficiently large.  We write $f \gg_K g$ if $c$ is allowed to depend on $K$.  The doubling constant of a set $A$ is $\frac{|A+A|}{|A|}$.  A set has \textit{small doubling} if its doubling constant is $O(1)$.  A \textit{generalized arithmetic progression} $G$ is a set $\{a+x_1 d_1 + \ldots x_k d_k : |x_i| \le L_i \}$; without loss of generality, we may assume $d_i > 0$ for all $i = 1, \ldots , k$;  we call $k$ the dimension of $G$; $|G|$ is the volume of $G$.  Moreover, $G$ is proper if the volume of $G$ is maximal -- $\prod_i(2L_i + 1)$.


\section{Condensing Lemma}
The following lemma in conjunction with \cref{sanders} will allow us to prove \cref{condenselemma}.
\begin{lemma}\label{convexgeometrylemma}
Let $G$ be a proper generalized arithmetic progression of the form 
\[ G := \{ \sum_{i=1}^k a_i d_i : |a_i| \le L_i \} \]
such that
\[ G' := \{ \sum_{i=1}^k a_i d_i : |a_i| \le 4L_i \} \]
is also a proper generalized arithmetic progression.  Then, there exists a constant $c = c(k)$, $d_1', \ldots, d_k'$, and a map $\phi$ with the following properties:
\begin{enumerate}
\item $\phi(\sum_{i=1}^k a_i d_i) = \sum_{i=1}^k a_i d_i'$.
\item $\phi$ is an order-preserving Freiman 2-isomorphism.
\item For any $x \in G$, $|\phi(x)| \le c|G|$.
\end{enumerate}
\end{lemma}

In order to prove this lemma we need some definitions and results from convex geometry, from which we refer the reader to \cite{bar02} as a reference.
\subsection{Convex Geometry Preliminaries}
Here, we review some basic notions and facts from convex geometry and linear algebra.  Our goal is to prove a version of Siegel's Lemma, \cref{siegel}.  The familiar reader is welcome to skip this section.  

A set $K \subset \mathbb{R}^n$ is said to be a convex cone if for all $\alpha, \beta \geq 0$ and $\textbf{x}, \textbf{y} \in K$ we have $\alpha \textbf{x} + \beta \textbf{y} \in K$. 
\begin{fact}\label{convexgeomfact1}
Let $a_{i,j} \in \mathbb{R}$.  Then, the set of solutions to the system of linear inequalities
\begin{equation}\label{systemfact} \sum_{i=1}^k a_{i,j} x_i > 0 \text{\hspace{1cm}} j = 1, \ldots , n \end{equation}
is a convex cone.
\end{fact}
\begin{proof}
Let $\textbf{x}$ and $\textbf{y}$ be solutions to the system of linear inequalities defined above and let $\alpha, \beta \ge 0$.  It is trivial to verify that $\alpha \textbf{x}$ and $\textbf{x}+\textbf{y}$ are also solutions to \eqref{systemfact}.
\end{proof}
For points $\textbf{x}_1, \ldots , \textbf{x}_m \in \mathbb{R}^n$ and non-negative real numbers $\alpha_1, ..., \alpha_m$, the point 
\[ \textbf{x} = \sum_{i=1}^m \alpha_i \textbf{x}_i \]
is called a conic combination of the points $\textbf{x}_1, ..., \textbf{x}_m$. The set $co(D)$ is defined as all conic combinations of points in $D \subset \mathbb{R}^n$ and is called the conic hull of the set $D$. For a non-zero $\textbf{x} \in \mathbb{R}^n$ the conic hull of $\textbf{x}$ is called a ray spanned by $\textbf{x}$.  A ray $R$ of the cone $K$ is called an extreme ray if whenever  $\alpha \textbf{x} + \beta \textbf{y} \in R$ for $\alpha > 0$, $\beta > 0$ and $\textbf{x}, \textbf{y} \in K$ then $\textbf{x}, \textbf{y} \in R$. An extreme ray is a 1-dimensional face of the cone.  A set $B \subset K$ is called a base of $K$ if $0 \notin B$ and for every point $\textbf{x} \in K$, $\textbf{x} \neq 0$, there is a unique representation $\textbf{x} = \lambda \textbf{y}$ with $\textbf{y} \in B$ and $\lambda > 0$.
\begin{fact}\label{convexgeomfact2}
For $i = 1, \ldots , k$ and $j = 1 , \ldots , \ell$, let $a_{i,j}  \in \mathbb{R}$.  If
\[ A := \{(x_1, \ldots , x_k) \in \mathbb{R}^k: \sum_{i=1}^k a_{i,j} x_i > 0 \text{ for $j = 1, \ldots , \ell$} \} \]
is nonempty with a solution in the positive quadrant of $\mathbb{R}^k$,  then the closure of $A$ has a compact base.
\end{fact}
\begin{proof}
Observe that $A$ is an open set, and since there is at least one solution, it is nonempty.  By \cref{convexgeomfact1}, $A$ is also a convex cone.  Let $cl(A)$ be the closure of $A$, and let $H:= \{ (x_1, \ldots , x_k) \in \mathbb{R}^k: x_1 + \ldots + x_k =1 \}$.  We claim that $B := cl(A) \cap H$ is a compact base of $cl(A)$.  Clearly $B$ is a subset of $cl(A) \setminus \{0\}$.  Let $\textbf{y} \in cl(A)$ and consider the line $\lambda \textbf{y}$.  Since $A$ is a convex cone, this line is contained in $cl(A)$ for all $\lambda \ge 0$.  If this line intersects $B$, then $B$ must be a compact base, but clearly it does at $\lambda = \frac{1}{y_1 + \ldots + y_k}$.  
\end{proof}
\begin{theorem}[Cor. 8.5 \cite{bar02}]\label{convexgeomfact3}
If $K$ is a convex cone with a compact base, then every point $\textbf{x} \in K$ can be written as a conic combination 
\[ \textbf{x} = \sum_{i = 1}^m \lambda_i \textbf{x}_i, \ \ \lambda_i \geq 0, \ \ i = 1, ..., m, \]
where the $\textbf{x}_i$ each span an extreme ray of $K$.
\end{theorem}
Lastly, we need the well-known linear algebraic result known as Cramer's rule.
\begin{theorem}[Cramer's Rule]
Let $A$ be a $k \times k$ matrix over a field $F$ with nonzero determinant.  Then, $A \textbf{x} = \textbf{b}$ has a unique solution given by
\[ \textbf{x}_i = \frac{det(A_i)}{det(A)} \text{ $i = 1, \ldots , k$} \]
where $A_i$ is obtained by replacing the $i$th column in $A$ with $\textbf{b}$.
\end{theorem}
We combine the above tools to prove a standard variant of Siegel's Lemma.
\begin{lemma}[Siegel's Lemma]\label{siegel}
Let $d_1, \ldots , d_k \in \mathbb{Z}^+$.  If the interior of the convex cone defined by the system of inequalities
\begin{equation}\label{system2} \{ \sum_{i = 1}^k a_i x_i > 0: a_1d_1 + \ldots + a_k d_k > 0; -M_i \le a_i \le M_i \} \end{equation}
is nonempty, then there exists a solution $(z_1, \ldots , z_k) \in \mathbb{Z}^k$ to the system satisfying
\[ |z_i| \le k! \prod_{j \neq i} M_j \]
for all $i = 1 , \ldots , k$.
\end{lemma}
\begin{proof}
Consider the solution space defined by the system of inequalities \eqref{system2}.  By \cref{convexgeomfact1} the solution space forms a convex cone.  Let $K$ be the closure of the cone defined by the inequalities in \eqref{system2}.  Since $d_i \in \mathbb{Z}^+$, $x_i > 0$ is one of our inequalities for all $i = 1, \ldots , k$. Additionally, by the supposition that there is a solution to \eqref{system2}, we may apply \cref{convexgeomfact2} to deduce that $K$ has a compact base.  Hence, we may apply \cref{convexgeomfact3} to conclude that each $\textbf{x}\in K$ can be represented as conic combinations of the points on its extreme rays.  
 
 Because all extreme rays have dimension 1 in a $k$-dimensional space, they must each be intersections of $k-1$ linearly independent hyperplanes.  Because they are the extreme rays corresponding to the system of inequalities \eqref{system2}, the intersecting hyperplanes must correspond to an equation
 \[ a_1 x_1 + \ldots + a_k x_k = 0. \]
For each extreme ray, we show how to find an integer point on it; then, taking a conic combination of these integer points will allow us to find an integer point in the interior of the cone.

Let the intersection of the following hyperplanes define one of our extreme rays: 
\begin{equation}\label{hyperplanes1} \{ a_{i,1} x_1 + \ldots + a_{i,k} x_k = 0 : i = 1, \ldots , k-1 \}. \end{equation}
This system of equations will have all the points along our extreme ray as a solution -- in other words, there are infinitely many solutions.  Hence, we may treat one of the variables $x_i$ as a free variable while the other variables depend on it.  Without loss of generality, assume that $x_k$ is the free variable, and let us solve the system for the case when $x_k = 1$.  We will use Cramer's rule.  Let
\[ \Delta := \begin{vmatrix}a_{1,1} & \ldots & a_{1,k-1} \\ a_{2,1} & \ldots & a_{2,k-1} \\ \vdots \\ a_{k-1, 1} & \ldots & a_{k-1,k-1} \end{vmatrix} \]
and let $\Delta_i$ be the determinant of the same matrix with the $i$th row and column replaced by $-a_{j,k}$ for $j=1, \ldots , k-1$:
\[ \Delta_i := \begin{vmatrix}a_{1,1} & \ldots & a_{1,i-1} & -a_{1,k} & a_{1,i+1} & \ldots & a_{1,k-1}  \\ \vdots & \ddots \\ a_{k-1,1} & \ldots & a_{k-1,i-1} & -a_{k-1,k} & a_{k-1,i+1} & \ldots & a_{k-1,k-1} \end{vmatrix}. \]
By Cramer's rule, the solution to the system is given by $x_i = \frac{\Delta_i}{\Delta}$ for $i= 1, \ldots , k-1$.  By instead choosing $x_k = c$ instead of $x_k = 1$, we see that we can require that any multiple of this is also a solution to \eqref{hyperplanes1}.  Hence, $(|\Delta_1|, \ldots , |\Delta_{k-1}|, |\Delta|)$ is an integer solution to our system that lies along our edge.  For convenience, let $\Delta_k := \Delta$.

Now, we may get such an integer solutions for each of our extreme rays.  Not all extreme rays belong to the same face since $K$ has interior points.  In particular, we may take a set of $k+1$ of such rays that do not all lie along the same face and get $k + 1$ integer solutions as we did above.  Call these solutions $\textbf{p}_1, \ldots , \textbf{p}_{k+1}$.  We can bound the entries of $\textbf{p}_i$ by using a trivial bound on the determinant of our matrices formed above.  We have that for $i=1, \ldots , k$, since each entry $|a_{i,j}| \le M_j$, the determinant is bounded as follows:
\[ |\Delta_i| \le k! \prod_{j \neq i} M_j \]
  Moreover, the sum, $\textbf{p}_1 + \ldots + \textbf{p}_{k+1} =: (d_1' , \ldots , d_k')$ does not belong to any of the faces of $K$; so, it belongs to the interior of the cone, and hence, satisfies \eqref{system2}.\end{proof}
The broad idea of the proof of \cref{convexgeometrylemma} is as follows.  We are given a generalized arithmetic progression $G := \{ \sum_{i=1}^k a_i d_i : - L_i \le y_i \le  L_i \}$.  In a sense, this can be identified with the point $(d_1, \ldots , d_k)$.  What we would like to find is another generalized arithmetic progression, $H := \{ \sum_{i=1}^k b_i d_i': -L_i' \le b_i \le L_i' \}$ which maintains the same additive structure as $G$, but is much more compact.  Viewed another way, we want to find a point $(d_1' , \ldots, d_k')$ much closer to the origin than $(d_1, \ldots , d_k)$ that also satisfies certain inequalities (these are what maintain the additive structure).  Hence, we reduce our problem to finding an integer solution, relatively close to the origin, to a set of linear inequalities.
\subsection{Proof of the Condensing Lemma}
The crux in the proof of the Condensing Lemma is to first prove it for generalized arithmetic progressions; that is, to first prove \cref{convexgeometrylemma}.
\begin{proof}[Proof of \cref{convexgeometrylemma}]
Given $G$ as in the statement of the Lemma, consider the following set of inequalities: 
\begin{equation}\label{system1} \{ \sum_{i=1}^k a_i x_i > 0 : a_1 d_1 + \ldots + a_k d_k > 0; -4 L_i \le a_i \le 4 L_i \}.
\end{equation}  
We will first prove that if $(d_1', \ldots , d_k')$ is an integer solution to the above system of inequalities, then the map $\phi:G \rightarrow \mathbb{Z}$ defined by 
\[ \phi \left (\sum_{i=1}^k a_i d_i \right )= \sum_{i=1}^k a_i d_i' \]
is an order-preserving Freiman 2-isomorphism.  Note that $\phi$ is well-defined since $G$ is proper.

To see that $\phi$ is order-preserving, if
\[ \sum_{i=1}^k a_i d_i < \sum_{i=1}^k b_i d_i \]
 for two elements in $G$, then 
\[ \sum_{i=1}^k (b_i - a_i) x_i > 0 \]
is one of the inequalities in \eqref{system1} that $(d_1', \ldots , d_k')$ must satisfy; so 
\[ \phi \left (\sum_{i=1}^k a_i d_i \right ) = \sum_{i=1}^k a_i d_i' < \sum_{i=1}^k b_i d_i' = \phi \left (\sum_{i=1}^k b_i d_i \right ). \]
For the converse, if
\begin{equation}\label{ordpres1} \sum_{i=1}^k a_i d_i' < \sum_{i=1}^k b_i d_i' \end{equation}
and 
\[ \sum_{i=1}^k (b_i - a_i) d_i \le 0, \]
then we get a contradiction as follows.  First, if 
\[ \sum_{i=1}^k (b_i - a_i) d_i = 0, \]
then $b_i = a_i$ because $G$ is a proper generalized arithmetic progression.  Hence, \eqref{ordpres1} cannot hold in this case.  If 
\[ \sum_{i=1}^k (b_i - a_i) d_i < 0, \text{ then } \sum_{i=1}^k (a_i - b_i) d_i > 0 \]
which implies that 
\[ \sum_{i=1}^k (a_i - b_i) x_i > 0 \]
is an inequality in \eqref{system1} satisfied by $(d_1', \ldots , d_k')$, again contradicting \eqref{ordpres1}.

If we have points in $G$ such that
\[ \sum_{i=1}^k a_i d_i + \sum_{i=1}^k b_i d_i  = \sum_{i=1}^k s_i d_i + \sum_{i=1}^k t_i d_i \]
then
\begin{equation}\label{freiman1} \sum_{i=1}^k (a_i + b_i) d_i = \sum_{i=1}^k (s_i + t_i) d_i. \end{equation}
Moreover, $|a_i + b_i|, |s_i + t_i| \le 2L_i$.  Hence, each side of \eqref{freiman1} corresponds to an element in $G'$, and by the fact that $G'$ is proper, we must have that $a_i + b_i = s_i + t_i$ for $i= 1, \ldots, k$.  This implies that indeed, $\phi$ is a Freiman 2-homomorphism:
\begin{equation}\label{freiman2} \sum_{i=1}^k a_i d_i' + \sum_{i=1}^k b_i d_i' = \sum_{i=1}^k s_i d_i' + \sum_{i=1}^k t_i d_i'.\end{equation}
For the converse, if \eqref{freiman2} holds and \eqref{freiman1} does not, then without loss of generality, we may assume
\[  \sum_{i=1}^k (a_i + b_i - s_i - t_i) d_i > 0. \]
However, $a_i + b_i - s_i - t_i \in [-4 L_i , 4L_i]$, and so the inequality 
\[ \sum_{i=1}^k (a_i + b_i - s_i - t_i) x_i > 0 \]
is satisfied by $(d_1', \ldots , d_k')$ which contradicts \eqref{freiman2}.  This proves $\phi$ is a Freiman 2-isomorphism.

Now, we bound the image of $\phi$.  To apply \cref{siegel}, Siegel's Lemma, we remind the reader that $(d_1, \ldots , d_k)$ is a solution to \eqref{system1}, and hence there exists a solution in the interior of the convex cone.  We also remind the reader that, by the definition of a generalized arithmetic progression, $d_i \in \mathbb{Z}^+$.  Let $(d_1', \ldots , d_k')$ be a solution to \eqref{system1} guaranteed by \cref{siegel}.  By the conclusion of \cref{siegel}, the image of $\phi$ is bounded as follows:
\[ \left |\phi \left ( \sum_{i=1}^k y_i d_i   \right ) \right | = \left |\sum_{i=1}^k y_i d_i' \right  | \le \left | \sum_{i=1}^k L_i  d_i' \right | \le \left | \sum_{i=1}^k L_i (k+1)( 4^k k! \prod_{j \neq i} L_j) \right |  \le (k +1)! 4^k \prod_{j=1}^k L_j. \]
So if $g \in G'$, $\phi(g) \in [- 4^k (k + 1)! |G|, 4^k (k + 1)!|G|]$.
\end{proof}

The proof of \cref{condenselemma} follows easily from applying \cref{sanders} to a set with small doubling.  
\begin{proof}[Proof of \cref{condenselemma}]
Let $A \subseteq \mathbb{Z}$ be such that $|A+A| \le K|A|$.  All constants $c_i$ in the following depend only on $K$.  We may apply \cref{sanders} to $A$ to get a proper, symmetric, generalized arithmetic progression $G$ with $A \subseteq G$, $|G| \ge c_1|A|$, dimension at most $c_2$.
Denote $G$ as 
\[ G = \{ u + \sum_{i=1}^k x_i d_i : |x_i| \le L_i \}. \]
Since the composition of an order-preserving Freiman isomorphism with a linear map (in this case, the map $\psi(x) = x - u$)  is also an order-preserving Freiman isomorphism, we may assume $u=0$, or simply work with the sets $\psi(A)$ and $\psi(G)$ instead.  Let 
\[ G' := \left \{ \sum_{i=1}^k x_i d_i :  |x_i| \le \lfloor L_i/4 \rfloor \right \}. \]
Apply \cref{convexgeometrylemma} to $G'$ to get an order-preserving Freiman isomorphism $\phi:G' \rightarrow [-c_3 |G'|, c_3 |G'|]$.  We have that $A \subseteq G$, but $A \cap G'$ may not be large.   However, by considering the $4^k$ different translates, $G' + v$, where $v = j \lfloor L_i/4 \rfloor$ for $j = 0, 1, 2, 3$, $i=1, \ldots , k$ there exists an integer $v$ such that 
\[ |A \cap (G' + v)| = |(A-v) \cap G'| \gg_k |A|. \]
Let $A' := A \cap (G'+v)$.  So, $\phi$ is an order-preserving Freiman isomorphism from $A'-v$ to $[-c_3 |G''|, c_3 |G''|]$.  The composition of an order-preserving Freiman isomorphism with the linear map $\psi'(x) = x - v$ is also an order-preserving Freiman isomorphism.  So, $\phi_0(x):= \phi(x) - v$ is an order-preserving Freiman isomorphism from $A'$ to $[-c_3 |G'| , c_3 |G'|]$.  
Since $|G| \le c_2 |A|$, we have $[-c_3 |G'| , c_4 |G'|] = [-c_4 |A'|, c_4|A'|]$, proving the lemma.
\end{proof}


\section{Indexed Energy}
We provide an interesting combinatorial application of the Condensing Lemma.  In \eqref{indexdef} and \eqref{energydef}, we defined the notions of indexed energy and additive energy.  One always has the following relationship between the additive energy and indexed energy: 
\[ |A|^2 \le EI(A) \le E(A) \le |A|^3. \]
If $A$ is an arithmetic progression the relationship is strengthened to $EI(A) = E(A)$.  Moreover, for an arithmetic progression $A$, $E(A)$ is maximized.  Thus, it is natural to wonder if one loosens the restriction to $E(A) \gg |A|^3$ then is $EI(A) \gg |A|^3$?  We provide a counterexample to show that this is false.
\begin{theorem}\label{counterexamplethm}
There exists an integer $N$ such that for every $n \ge N$, there exists $A \subset [n]$ such that $|A| \ge n/3$, $E(A) \ge \frac{1}{6}|A|^3$, and $EI(A) \le 2000|A|^2(\log{|A|})^2$.
\end{theorem}
Thus, one can indeed have the additive energy $\Omega(|A|^3)$ while the indexed energy is $O((|A|\log{|A|})^2)$.  However, when the additive energy is large, it turns out that one can still pass to a large subset $A' \subseteq A$, $|A'| = \Omega(|A|)$, which has indexed energy $\Omega(|A'|^3)$.  We note that when passing to a subset, the subset does not inherit the same indices as the superset, but rather it is reindexed in the natural way.  Hence, $EI(A')$ is not bounded from above or below by $EI(A)$.
\begin{theorem}\label{mainthm}
For any $K > 0$, there exists $c_1, c_2$ dependent only on $K$ such that if $A$ is a finite set of integers with $|A+A| \le K|A|$ then the following holds.  There exists an $A' \subseteq A$ such that $EI(A') \ge c_1 |A'|^3$ and $|A'| \ge c_2 |A|$.
\end{theorem}
The condition that $|A+A| \ll_K |A|$ may easily be loosened to $E(A) \gg_K |A|^{3}$ by applying the following well-known result of Balog-Szemer\'edi \cite{balsze94} and Gowers \cite{gow98} to pass to a subset with small doubling.
\begin{theorem}[Balog-Szemer\'edi\cite{balsze94}, Gowers\cite{gow98}]\label{bsg}
For any $K > 0$, there exists $c_1, c_2$ such that if $A \subseteq \mathbb{Z}$ is such that $E(A) \ge K|A|^3$ then there exists $A' \subseteq A$ with $|A'|\ge c_1 |A|$ and $|A'+A'| \le c_2 |A'|$.
\end{theorem}


\subsection{Indexed energy in subsets of $[1,n]$}
It turns out that if $A$ is a dense subset of an interval, then there is a simple algorithm that can find a subset $A' \subseteq A$ with $|A'| \gg |A|$ and $EI(A') \gg |A'|^3$.  Thus, the general case may then be quickly deduced by applying the Condensing Lemma.  We first begin with a lemma that states, loosely speaking, that if $A$ is a dense subset of $[1,n]$, then one can choose a large subset $A' \subseteq A$ that is equidistributed over the interval. 
\begin{lemma}\label{goodspacing}For every $\delta > 0$, there exists $c_1, c_2, c_3, N$ such that if $A \subseteq [1,n]$ with $n > N$ and $|A| = \delta n$, then the following holds.  There exists an $A' \subseteq A$, $|A'| \ge c_1 |A|$ and for $c_3 |A|$ elements $x \in A'$, we have that 
\begin{align}\label{equid} x \in [(j-1)c_2, j c_2) \text{ and } |\{ y \in A': y < x\}| = j-1 \end{align}
\end{lemma}
It is easy to establish that a set with property \eqref{equid} has large indexed energy.
\begin{lemma}\label{denseinterval}For every $\delta > 0$, there exists $c_0, c_1, N$ such that if $A \subseteq [1,n]$ with $n > N$ sufficiently large and $|A| = \delta n$, then $A$ has a subset $A' \subseteq A$ with $|A'| \ge c_1 |A|$ and $EI(A') \ge c_0 |A|^3$.
\end{lemma}
\begin{proof}[Proof of \cref{goodspacing}]
Denote $A = \{a_1 < a_2 < \ldots < a_{\delta n} \}$.  Let $d = \lfloor \frac{2}{\delta} \rfloor$.  Let $I_j = [(j-1)d, jd)$ for all $j = 1 , \ldots , \lceil \frac{n}{d} \rceil$.  Let $A_j = A \cap I_j$, and observe that the $A_j$ are pairwise disjoint -- a fact that will be important later when we estimate a union.  We pick our subset $A'$ as follows:
\begin{itemize}
\item \textit{Step 1:}  If $A_1 \neq \emptyset$ then let $X_1 = \{ a_1 \}$.  Else, $X_1 = \emptyset$.
\item \textit{Step k:}   For $k = 2, \ldots , \lceil \frac{n}{d} \rceil$, if $|A_{k} \cup X_{k-1}| \le k$, then $X_k := A_{k} \cup X_{k-1}$.  Else, arbitrarily choose $Y \subseteq A_k$ so that $|Y \cup X_{k-1}| = k$ and then let $X_k := Y \cup X_{k-1}$.
\end{itemize}
Let $A' = X_{\lceil \frac{n}{d} \rceil}$.  
To prove that $A'$ satisfies the conclusion of the lemma, we analyze the algorithm as follows.  First, note that $X_1 \subseteq X_2 \subseteq \ldots \subseteq X_{\lceil \frac{n}{d} \rceil} = A'$ and $|X_i| \le i$ for all $i$.  Now, the sets $X_j$ for which $|X_j| = j$ we will call sated, and the others we will call hungry.  Note that if $X_j$ is sated, then $|X_{j-1}| \le j-1$, and $|X_j| = j$; hence, there is an $x \in X_j \setminus X_{j-1}$ such that $|\{y \in A': y < x \}| = j-1$.  Showing that lots of $X_j$ are sated will prove the lemma.  Let $J = \{ j_1, j_2, \ldots , j_m \}$ be the set of indices such that $X_{j_i}$ is sated.  Observe that for indices between $j_i$ and $j_{i+1}$, we must not have enough elements to make any of those corresponding sets sated.  More precisely, for all $1 \le k \le j_{i+1} - j_i - 1$,
\[ |A_{j_i + k}| \le k - 1 - \sum_{s=1}^{k-1} |A_{j_i + s}|. \]
This implies that
\[ \left | \bigcup_{k = 1}^{j_{i+1} - j_i - 1} A_{j_i + k} \right | \le j_{i+1} - j_i - 1. \]
Note that in the case that we have two consecutive sated sets, that is $j_{i+1} = j_i + 1$, we define $\cup_{k=1}^0 A_{j_i + k} = \emptyset$, and the inequality still holds.  We can now bound the total number of elements in hungry sets (except for potential hungry sets before $j_1$ or after $j_m$) by taking the union as follows:
\[ \left | \bigcup_{i = 1}^{m-1} \bigcup_{k = 1}^{j_{i+1} - j_i -1} A_{j_i + k} \right | \le \sum_{i = 1}^{m-1} j_{i+1} - j_i - 1 = j_m - j_1 - (m-1) \]
We must also account for the hungry sets occurring before $j_1$.  They contain at most $j_1 - 2$ elements.  The hungry sets occurring after $j_m$ contain at most $\lceil \frac{n}{d} \rceil - j_m$ elements.  Hence, the total number of elements in $A$ that appear in hungry sets is
\[ j_m - j_1 - (m - 1) + j_1 - 2 + \left \lceil \frac{n}{d} \right \rceil - j_m \le \left \lceil \frac{n}{d} \right \rceil \]
Thus, we have at least $\delta n - \lceil \frac{n}{d} \rceil \ge \delta n / 4$ elements of $A$ are distributed over the intervals where the corresponding $X_j$ are sated.  Since each interval is of length $d$, it contains at most $d$ elements of $A$.  Then we must have that $m$, the number of sated sets, is at least 
\[ m \ge \frac{1}{d} \cdot \frac{\delta n}{4} \ge \frac{n \delta^2}{8} . \]
This in turn gives us a lower bound on $|A'| = j_m \ge m \ge \frac{n\delta^2}{8} = \frac{\delta}{8}|A|$. 
\end{proof}
\begin{proof}[Proof of \cref{denseinterval}]
%
Apply \cref{goodspacing} to $A$ to get $A', c_1, c_2, c_3$ as in the lemma.  Denote $A' = \{ b_1 < b_2 < \ldots < b_m\}$, and so $m = |A'|$.  Let $J_0$ be the set of integers $j$ such that there exists an $x \in A'$ where \eqref{equid} holds.  At least half of $J_0$ is either even or odd; without loss of generality, assume at least half are even and let $J := \{ j \in J_0 : \text{$j$ is even} \}$.  We know that $|J| \ge 1/2|J_0| \ge c_3/2|A|$.  Since $EI(A') \ge | \{(i,j,k,l) \in J^4 : b_i + b_j = b_k + b_l \text{ and } i + j = k + l \} |$, we will simply work with these quadruples.

For all of the following, $b_j \in A'$ will be assumed to have $j\in J$.  Let $t \in \{2, \ldots , 2m\}$, and define
\[ r_{J+J}(t) := | \{ (i,j) \in J \times J: i + j = t\} |. \]   Observe that for pairs $(i,j) \in J \times J$, we have that if $i+j = t$ then $b_i + b_j \in [(t-2)c_2, tc_2)$.  Additionally, since $J$ is only the set of even indices, if $j \in J$, then $j-1 \notin J$ and $j+1 \notin J$.  Hence, if $b_i + b_j \in [(t-2)c_2, tc_2)$, then we can deduce $i+j = t$.  Observe that there are only $2c_2$ values that $b_i + b_j$ can take when $i+j$ is fixed.  For every $k \in [0, 2c_2-1]$, define
\[ t_k := | \{(i,j) \in J \times J : b_i + b_j = (t-2)c_2 + k \} |. \]
We can bound the indexed energy of $A'$ by two applications of Cauchy-Schwarz as follows:
\begin{align*} EI(A') \ge \sum_{t=2}^{2m} \sum_{k=0}^{2c_2-1} t_k^2  \ge \sum_{t=2}^{2m} \frac{1}{2c_2} \left ( \sum_{k=0}^{2c_2-1} t_k \right )^2 
& = \sum_{t=2}^{2m} \frac{1}{2c_2} r_{J + J}(t)^2
\\ & \ge \frac{1}{4c_2m} \left ( \sum_{t=2}^{2m} r_{J+J}(t) \right )^2
\\ & = \frac{|J|^4}{4c_2m} \ge c_0 |A|^3.
\end{align*}
for some constant $c_0$ depending only on $\delta$.
\end{proof}

Now, we are ready to prove \cref{mainthm}.
\begin{proof}[Proof of \cref{mainthm}]
%
Let $A$ be a finite subset of integers with $|A+A| \le c|A|$.  All constants $c_i$ in the following depend only on $c$.  Apply \cref{condenselemma} to $A$ to get a set $A' \subseteq A$ with $|A'| \ge c_1 |A|$ and an order-preserving Freiman $\phi:A' \rightarrow [-c_2|A'|, c_2|A'|]$.  We may assume at least one third of the elements are in $[1, c_2|A'|]$ or simply shift $A'$ by $v = c_2|A'|$.  Apply \cref{denseinterval} to $\phi(A')$ to conclude that $EI(\phi(A')) \ge c_3 |\phi(A')|^3 = c_3|A'|^3$.  It is easy to see that $EI(\phi(A')) = EI(A')$ since $\phi$ is an order-preserving Freiman 2-isomorphism, so the result follows.
\end{proof}

\subsection{An Extremal Construction}
The proof of \cref{counterexamplethm} follows from the following lemma.
\begin{lemma}\label{counterexamplelemma}
Let $n \in \mathbb{N}$, and let $p \in (1,2)$ and denote $p = 1+\epsilon$.  Let $A = \{ \lfloor a^p \rfloor : 1 \le a \le \lfloor n^{1/p} \rfloor \}$.  Then, $EI(A) \le 16 \epsilon^{-1} n^2 \log{n}$.  
\end{lemma}
\begin{proof}[Proof of \cref{counterexamplelemma}]
Let $x, y \in [1,\lfloor n^{1/p} \rfloor]$ with $x +1 < y$.  The main part of the argument is to establish the following bound:
\begin{equation}\label{diffeqn1} x^p + y^p - (x+1)^p - (y-1)^p > \frac{\epsilon(y-x)}{2y} \end{equation}
For now, assume \eqref{diffeqn1} holds.  If $x + y = z + w$, then by convexity, $x^p + y^p \neq z^p + w^p$ unless $z=x$ and $y=w$ or vice versa.  However, it may happen that $x+y=z+w$ and $\fl{x^p} + \fl{y^p} = \fl{z^p} + \fl{w^p}$.  Since $\fl{a^p} = a^p - [a^p]$, where $[a^p]$ is the noninteger part of $a^p$, we must have that if $x+y=z+w$ and
\[ \fl{x^p} + \fl{y^p} = \fl{z^p} + \fl{w^p} \]
then
\[ |x^p + y^p - z^p - w^p| < 2. \]
So, fixing an $x$ and a $y$, we can bound how many other pairs $z$ and $w$ can have $z+w = x+y$ and $\fl{z^p} + \fl{w^p} = \fl{x^p} + \fl{y^p}$.  More specifically, we find the largest $t$ such that
\[ x^p + y^p - (x+t)^p - (y-t)^p  < 2. \]
Using \eqref{diffeqn1}, the triangle inequality, and letting $k = y-x$ we get that
\[ x^p + y^p - (x+t)^p - (y-t)^p \ge \frac{\epsilon k}{2y} + \frac{\epsilon (k+2)}{2(y-1)} + \ldots + \frac{\epsilon(k+2(t-1))}{2(y-(t-1))} \]
Each term in the sum is greater than or equal to $\frac{\epsilon k }{2y}$, so we get a lower bound of $\frac{t\epsilon k }{2y}$.  So, if $t \ge \frac{4y}{\epsilon(y-x)}$, then we cannot have 
\[ \fl{x^p} + \fl{y^p} = \fl{(x+t)^p} + \fl{(y-t)^p}. \]
This allows us to conclude that any quadruple $(x,y,z,w)$ with $x+y = z+w$, with $x < z < w < y$, $z < x < y < w$, $w < y < x < z$, or $y < w < z < x$ we must have that $|z-x| < \frac{4y}{\epsilon(y-x)}$.  Accounting for an extra factor of $2$ for when $x<w < z < y$ and so on, we can bound the indexed energy of $A$
\[ EI(A) \le 2 \sum_{y} \sum_{x < y} \frac{4y}{\epsilon(y-x)} \]
Estimating this summation by using the harmonic series gets us that
\[ EI(A) \le \frac{16}{\epsilon}n^2 \log{n} \]
concluding the proof assuming that \eqref{diffeqn1} holds.

Now, we work to establish \eqref{diffeqn1}.
First, since $f(x):=x^p$ is convex for $p > 1$, it is easy to establish the following bound for any $x > 0$: 
\begin{equation}\label{convex} p(x+1)^{p-1} > (x+1)^p - x^p  > px^{p-1} \end{equation}
Assuming $p = 1+\epsilon < 2$, we have that $g(x) := x^{p-1}$ is concave.  Doing a similar analysis for $g(x)$, we get that for any $\ell \ge 1$
\begin{equation}\label{concave} 
(x+\ell)^{p-1} - x^{p-1} > \ell(p-1)(x+\ell)^{p-2}, \end{equation}
Using \eqref{convex}, we have
\[ x^p + y^p - (x+1)^p - (y-1)^p = \]
\[ = y^p - (y-1)^p - ((x+1)^p - x^p) > p(y-1)^{p-1} - p(x+1)^{p-1} \]
Using \eqref{concave} and reminding the reader that $k := y-x > 1$, we establish \eqref{diffeqn1}
\[ p[(y-1)^{p-1} - (y-k+1)^{p-1}] > p[(k-2)(p-1)(y-1)^{p-2}] > \frac{\epsilon k}{2y}. \]
\end{proof}
\cref{counterexamplethm} follows by letting $\epsilon = \frac{1}{\log{n}}$.
\begin{proof}[Proof of \cref{counterexamplethm}]
Let $A$ be as in the above lemma, let $\epsilon = \frac{1}{\log{n}}$.  Then, for $n$ sufficiently large
\[ |A| = \lfloor n^{\frac{1}{1+\epsilon}} \rfloor = \left \lfloor n^{\frac{1}{1+\frac{1}{\log{n}}}} \right \rfloor = \left \lfloor \frac{n}{e} \cdot e^{\frac{1}{1+\log{n}}} \right \rfloor \ge \left \lfloor \frac{n}{e} \right \rfloor \ge \frac{n}{3}. \]
So, $A \subseteq [1,n]$, $|A| \ge \frac{n}{3}$, and $A+A \subseteq [1,2n]$.  Thus, $|A+A| \le 2n \le 6 |A|$.  Hence,
\[ E(A) \ge \frac{|A|^4}{|A+A|} \ge \frac{|A|^3}{6}. \]
On the other hand, by \cref{counterexamplelemma}, for $A$ sufficiently large,
\[ EI(A) \le 16 n^2 (\log{n})^2 \le 16 \cdot (9|A|)^2 (\log{9|A|})^2 \le 1296 |A|^2 (\log{9|A|})^2 \le 2000 |A|^2 (\log{|A|})^2. \]
\end{proof}

\section{Further Applications and Conjectures}

Since $|(A \times B) + (A \times B)| = |A+A||B+B|$, it is obvious that if $|A+A| \le K|A|$ and $|B+B| \le K|B|$, then for any $C \subseteq A \times B$ of size $\Omega( |A||B| )$, one has $|C+C| \ll_K |C|$.  However, if $|C| = O(\sqrt{|A||B|})$, one has little control of $|C+C|$.  Does there exist a $C \subseteq A \times B$ with $|C| = O(\sqrt{|A||B|})$, and $|C+C| \ll_K |C|$?  Clearly one could simply take $C = \{(a,b): a\in A\}$ for a fixed $b \in B$.  If we forbid such sets lying on vertical or horizontal lines by additionally requiring that for any distinct $(x,y),(z,w) \in C$ we have $(x-z)(y-w) > 0$, the answer is not as obvious.

For a set $C \subseteq A_1 \times \ldots \times A_k$, call $C$ a \textbf{diagonal set} if for any distinct pairs of elements $(x_1, \ldots , x_k), (y_1, \ldots , y_k) \in C$, one has $x_i - y_i > 0$ for all $i$ or $x_i - y_i < 0$ for all $i$.  Moreover, we call $C$ \textbf{truly diagonal} if there exists an $X$ and a tuple $(t_1, \ldots , t_k)$ such that $C = \{ (x, \ldots , x) - (t_1, \ldots , t_k) : x \in X\}$.  Clearly a truly diagonal set is also a diagonal set.   
\begin{theorem}\label{easyindexenergy}
For any $k, K \in \mathbb{N}$, there exists $c_1, c_2$ such that the following holds.  Let $A_1, \ldots , A_k \subseteq \mathbb{Z}$ be sufficiently large sets of size $n$ such that $|A_i + A_i| \le K|A_i|$ for all $i = 1, \ldots , k$.  Then, there exists a truly diagonal set $C \subset A_1 \times \ldots \times A_k$ such that $|C+C| \le c_1|C|$ and $|C| = c_2 (|A_1| \ldots |A_k|)^{1/k} = c_2 n$.  
\end{theorem} 
\begin{proof}
We may apply the Condensing Lemma to each $A_i$ individually to find constants $c_{1,i}, c_{2,i}$ depending on $K$ such that there exists a subset $A_i' \subseteq A_i$ and an order-preserving Freiman isomorphism to a set $B_i \subseteq [0, c_{1,k}n]$ with $|A_i'| \ge c_{2,i} n$.  Let $c_1$ be the maximum of $\{c_{1,i}: i = 1, \ldots , k\}$ and let $c_2$ be the minimum of $\{ c_{2,i}: i = 1, \ldots , k \}$.  So, we may view all the $B_i$ as being dense in the interval $[0,c_1 n]$.  Next, we claim that there exists $t_1, \ldots , t_k \in \mathbb{Z}$ such that
\[ \left | \bigcap_{i=1}^k (B_i + t_i) \right | \ge \frac{c_2^k}{(2c_1)^{k-1}} n. \]
We prove this by induction on $k$.  For $k = 1$, it is trivial.  For the induction step, let $X, Y \subset [1,c_1 n]$ be of size $\delta_1 n$ and $\delta_2 n$ respectively.  Then,
\[ \sum_{t = -(c_1 n-1)}^{c_1 n-1} |(X+t) \cap Y| = |X||Y| = \delta_1 \delta_2 n^2. \]
Hence, there exists a $t$ such that
\[ |(X+t) \cap Y| \ge \frac{\delta_1 \delta_2}{2 c_1} n. \]
Letting $X := B_{k}$ and $Y := \cap_{i=1}^{k-1} B_i + t_i$ finishes the inductive argument.  Now, let $C' := \cap_{i=1}^k B_i + t_i$, and denote $C' := \{ x_1 < \ldots < x_m \}$.  We let $C$ be the following set:
\[ C:= \{ (x_i - t_1, x_i - t_2 , \ldots , x_i - t_m): i = 1, \ldots , m \}. \]
Since $x_i - t_j \in B_j$, we have that $C \subseteq B_1 \times \ldots \times B_k$.  Since $x_i - t_j > x_{\ell} - t_j$ for $i > \ell$, $C$ must be diagonal.  Also, $|C| = |C'| \ge \frac{c_2^k}{(2c_1)^{k-1}}$.  Lastly, it is easy to see that
\[ |C+C| = |C'+C'| \le 2n = \frac{2^k c_1^{k-1}}{c_2^k} |C'|. \] 
\end{proof}
Although the above application is similar in spirit to the indexed energy problem -- letting $A \times B := A \times [1,|A|]$, where $B$ is the set of indices -- there are several subtle differences.  Mainly, in the indexed energy problem, when we pass to a subset, we are forced to reindex the set in a very specific way.  Therefore, this problem is related to, but does not imply \cref{mainthm}.  The following conjecture however would be general enough to imply \cref{mainthm}.
\begin{conjecture}\label{genindenergy}
Let $A, B \subseteq \mathbb{Z}$ be sets of size $N$ such that $|A+A|,|B+B| \le KN$.  Then, there exists $c_1, c_2$ depending only on $K$ such that the following holds.  There exists an $A' \subseteq A$ with $|A'| \ge c_1 |A|$, and if we denote $A' := \{ a_1' < \ldots < a_k' \}$ and $B := \{ b_1 < \ldots < b_n \}$, then
\[ |\{(a_i', a_j', a_k', a_{\ell}') : a_i' + a_j' = a_k' + a_{\ell}' \text{ and } b_i + b_j = b_k + b_{\ell} \} | \ge c_2 |A'|^3. \] 
%
\end{conjecture}
\cref{genindenergy} is true in the case where $B = [1,N]$ (or any arithmetic progression of size $N$) since this then becomes the indexed energy result.  It would be interesting to know whether the conjecture is even true in the case where $B$ is a generalized arithmetic progression of dimension 2.

Another problem closely related to the indexed energy problem is as follows.  Let $A \subseteq \mathbb{Z}$ and let $f:A \rightarrow \mathbb{Z}$ be such that $|f(A) + f(A)| \le c|A|$, and $|A+A| \le c|A|$.  Let $E_f$ denote the additive energy of the graph of $f$.  More precisely,
\[ E_f(A) := \{ (a,b,c,d): a+b = c+d, f(a) + f(b) = f(c) + f(d) \}. \]
When $f$ is the indexing function for a set $A$, $E_f(A)$ becomes $EI(A)$.  What is the relation between $E_f(A)$ and $E(A)$?  Here, we point out to the reader a subtle but important difference between this problem and the indexed energy problem: when passing to a subset, there is a natural way to reindex a set which is distinctly different than how a function restricted to a subset behaves.  Therefore, $E_f$ is not simply a generalization of $EI$.  Due to this lack of reindexing, there is not always an $A' \subseteq A$ with $E_f(A') \gg_K |A|^3$ when $E(A) \ge K|A|^3$.  For instance, let $f$ be the indexing function, let $A$ be as in \cref{counterexamplethm}, and since sets are not reindexed
\[ E_f(A') \le EI(A) \ll_K |A|^2\log{|A|}. \]
Moreover, $|\{(a + a', f(a) + f(a')): a,a' \in A\}| \gg |A|^2/\log{|A|}$.  As an openended question, we ask if there are any reasonable conditions that we can impose on $f$ or $A$ to arrive at a different conclusion?

Lastly, we remark that the content of \cref{goodspacing} is making a statement about equidistribution of a set in an interval.  This has been a well-studied topic in discrepancy theory; however, we are not aware of it appearing in this specific, combinatorial form -- where one is allowed to pass to a subset of the original set, and one only requires that for lots of interval, the subset is well-distributed.  We tepidly conjecture a generalization of \cref{goodspacing} to higher dimensions, but it would also be interesting if a counterexample was found.
\begin{conjecture}
Let $A \subseteq [1,n] \times [1,n]$ be of size $|A| = \delta n^2$.  There exists constants $c_1, c_2, c_3$ depending only on $\delta$ such that the following holds.  There exists an $A' \subseteq A$ such that $|A'| \ge c_1 |A|$ and for $c_2 n^2$ pairs $0 \le i,j \le n/c_3$, $|A' \cap [0, i c_3) \times [0, j c_3)| = ij$. 
\end{conjecture}

\noindent \textbf{Acknowledgment: } The authors would like to thank the anonymous referee for a careful reading of the paper that resulted in many helpful comments, corrections, and suggestions that greatly improved the presentation of the paper.

\bibliographystyle{plain}
\bibliography{C:/Users/Albert/Desktop/TeX/bigaddcomb}

%
%
%
%
%
%


\end{document}